\documentclass[12pt]{amsart}

\usepackage{psfrag}
\usepackage{color}
\usepackage{tikz}
\usetikzlibrary{matrix,arrows}
\usepackage{graphicx,graphics}
\usepackage{fullpage,amssymb,amsfonts,amsmath,amstext,amsthm,amscd,verbatim,enumerate}
\usepackage[T1]{fontenc}

\begin{document}

\newtheorem{theorem}{Theorem}[section]
\newtheorem{result}[theorem]{Result}
\newtheorem{fact}[theorem]{Fact}
\newtheorem{example}[theorem]{Example}
\newtheorem{conjecture}[theorem]{Conjecture}
\newtheorem{lemma}[theorem]{Lemma}
\newtheorem{proposition}[theorem]{Proposition}
\newtheorem{corollary}[theorem]{Corollary}
\newtheorem{facts}[theorem]{Facts}
\newtheorem{props}[theorem]{Properties}
\newtheorem*{thmA}{Theorem A}
\newtheorem{ex}[theorem]{Example}
\theoremstyle{definition}
\newtheorem{definition}[theorem]{Definition}
\newtheorem{remark}[theorem]{Remark}
\newtheorem*{defna}{Definition}

\newcommand{\notes} {\noindent \textbf{Notes.  }}
\newcommand{\note} {\noindent \textbf{Note.  }}
\newcommand{\defn} {\noindent \textbf{Definition.  }}
\newcommand{\defns} {\noindent \textbf{Definitions.  }}
\newcommand{\x}{{\bf x}}
\newcommand{\z}{{\bf z}}
\newcommand{\B}{{\bf b}}
\newcommand{\V}{{\bf v}}
\newcommand{\T}{\mathbb{T}}
\newcommand{\Z}{\mathbb{Z}}
\newcommand{\Hp}{\mathbb{H}}
\newcommand{\D}{\mathbb{D}}
\newcommand{\R}{\mathbb{R}}
\newcommand{\N}{\mathbb{N}}
\renewcommand{\B}{\mathbb{B}}
\newcommand{\C}{\mathbb{C}}
\newcommand{\ft}{\widetilde{f}}
\newcommand{\dt}{{\mathrm{det }\;}}
 \newcommand{\adj}{{\mathrm{adj}\;}}
 \newcommand{\0}{{\bf O}}
 \newcommand{\av}{\arrowvert}
 \newcommand{\zbar}{\overline{z}}
 \newcommand{\xbar}{\overline{X}}
 \newcommand{\htt}{\widetilde{h}}
\newcommand{\ty}{\mathcal{T}}
\renewcommand\Re{\operatorname{Re}}
\renewcommand\Im{\operatorname{Im}}
\newcommand{\tr}{\operatorname{Tr}}

\newcommand{\ds}{\displaystyle}
\numberwithin{equation}{section}

\renewcommand{\theenumi}{(\roman{enumi})}
\renewcommand{\labelenumi}{\theenumi}

\title{Epicycloids and Blaschke products}
\date{\today}
\author{Chunlei Cao, Alastair Fletcher, \and Zhuan Ye}
\footnotetext[2]{The first author was supported by China Scholarship Council and Beijing Institute of Technology.}

\begin{abstract}
It is well known that the bounding curve of the central hyperbolic component of the Multibrot set in the parameter space of unicritical degree $d$ polynomials is an epicycloid with $d-1$ cusps. The interior of the epicycloid gives the polynomials of the form $z^d+c$ which have an attracting fixed point. We prove an analogous result for unicritical Blaschke products: in the parameter space of degree $d$ unicritical Blaschke products, the parabolic functions are parameterized by an epicycloid with $d-1$ cusps and inside this epicycloid are the parameters which give rise to elliptic functions having an attracting fixed point in the unit disk. We further study in more detail the case when $d=2$ in which every Blaschke product is unicritical in the unit disk.
\end{abstract}

\maketitle

\section{Introduction}

\subsection{Unicritical polynomials}

Complex dynamics deals with the iteration of holomorphic and meromorphic functions. Very often whole families of mappings are considered and these families are given by parameterization. For example, Douady and Hubbard initiated the current activity of research in complex dynamics in the 1980s by a systematic study of quadratic polynomials of the form $f_c(z)=z^2+c$. This family is parameterized by the complex variable $c$, and moreover the dynamics of $f_c$ depends crucially on $c$: if $c$ lies in the Mandelbrot set $\mathcal{M}$, then the Julia set of $f_c$ is connected whereas otherwise the Julia set of $f_c$ is totally disconnected. For the standard definitions in complex dynamics, such as Julia set and Fatou set, we refer to, for example, \cite{CG}.

A polynomial of degree $d$ with exactly one critical point is called unicritical. By conjugating the polynomial, we may move the critical point to $0$ and obtain a function of the form $g_c(z) = z^d+c$. The structure of the Julia set depends on whether or not the critical point lies in the escaping set $I(g_c)$ and this dichotomy gives rise to the Multibrot set $\mathcal{M}_d$. The parameter $c$ is in $\mathcal{M}_d$ if and only if $0\notin I(g_c)$ if and only if $J(f_c)$ is connected, and otherwise $J(f_c)$ is totally disconnected.

The central hyperbolic component of the Multibrot set $\mathcal{M}_d$ consists of parameters for which $g_c$ has an attracting fixed point. Clearly $c=0$ is always in this component. The following result is well-known, see for example \cite[p.156]{BF}.

\begin{proposition}
\label{prop:multibrot}
The boundary of the central hyperbolic component of $\mathcal{M}_d$ is an epicycloid with $d-1$ cusps.
\end{proposition}

The only proof we were able to find of this result is in a paper of Geum and Kim \cite{GK} which may not be well-known, and so later we will provide a proof of Proposition \ref{prop:multibrot} for the convenience of the reader.

An epicycloid is a plane curve obtained by tracing the path of a particular point on a circle (a small one)  as it is rolled around the circumference of another circle (a big one), see for example Figure \ref{fig:1}. If the smaller circle has radius $r$ and the larger circle has radius $kr$, then in rectangular coordinates, when the center of the larger circle is at the origin, then the epicycloid could be given by the parametric equations
\[
x(\theta) = r\cos ((k+1)\theta) - r(k+1) \cos\theta ,\quad y(\theta) = r\sin((k+1)\theta)- r(k+1)\sin \theta,
\]
or in a more concise form
\begin{equation}\label{eq:epi} 
z(\theta) = r(e^{i(k+1)\theta} - (k+1)e^{i\theta}), \qquad \mbox{for} \quad \theta \in [0, \ 2\pi].
\end{equation}
Here, the smaller circle starts rolling on the left side of the bigger circle counterclockwise and $\theta$ represents the angle between the negative  $x$-axis and the line segment between the origin and  the center of the smaller circle (not the argument of $z(\theta)$).
In particular, when $k=1$, it is a cardioid; and when $k=2$, it is a nephroid.

\begin{figure}[h]
\begin{center}
\includegraphics[width=2in]{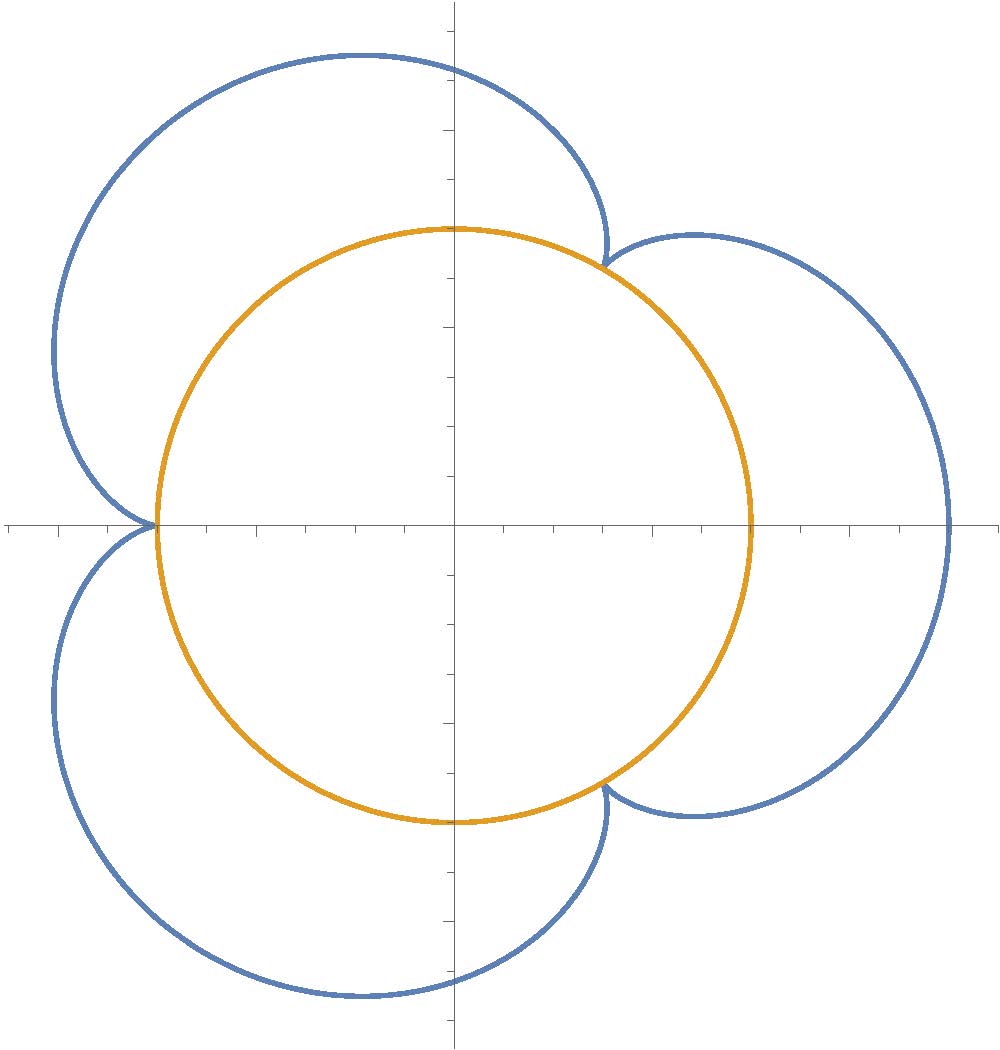}
\caption{An example of an epicycloid with $r=1$ and $k=3$}\label{fig:1}
\end{center}
\end{figure}

\subsection{Blaschke products}

In this paper, we will prove an analogous result in the setting of Blaschke products. A finite Blaschke product $B$ of degree $d$ is a function of the form
\[ B(z) = e^{it} \prod_{i=1}^d \left ( \frac{z-w_i}{1-\overline{w_i}z} \right ),\]
where $t \in [0,2\pi)$ and $w_i \in \D$ for $i=1,\ldots,d$. A finite Blaschke product is called non-trivial if it is not a M\"obius transformation, that is, if $d\geq 2$.

It is not hard to see that the unit disk $\D$, the unit circle $\partial \D$ and the complement of the closed unit disk $\overline{\C} \setminus \overline{\D}$ are all completely invariant sets for $B$. It is well-known that every surjective finite degree mapping $f:\D \to \D$ is in fact a finite Blaschke product and so it is natural to view finite Blaschke products as hyperbolic polynomials. Blaschke products also have symmetry with respect to $\partial \D$: if $z^* =1/\overline{z}$ denotes the reflection of $z$ in the unit circle, then
\begin{align*}
[B(z^*)]^* &= \left [ e^{it} \prod_{i=1}^d \left ( \frac{z^*-w_i}{1-\overline{w_i}z^*} \right )\right ]^*
= e^{it} \prod_{i=1}^d \left ( \frac{1-w_i\overline{z^*}}{\overline{z^*} - \overline{w_i}} \right )\\
&= e^{it} \prod_{i=1}^d \left ( \frac{1-w_i/z}{1/z-\overline{w_i}} \right ) = B(z).
\end{align*}

There is a classification of finite Blaschke products in analogy with M\"obius transformations. This classification relies on the Denjoy-Wolff Theorem which in our situation states that if $B$ is a non-trivial finite Blaschke product, then there exists a unique $z_0 \in \overline{\D}$ such that $B^n(z) \to z_0$ for every $z\in \D$. This point is called the Denjoy-Wolff point of $B$. Note that by the symmetry of Blaschke products, $z_0^*$ is also fixed by $B$, but we will be restricting to the unit disk, and then the Denjoy-Wolff point is unique.
\begin{enumerate}[(i)]
\item $B$ is called {\it elliptic} if the Denjoy-Wolff point $z_0$ of $B$ lies in $\D$. In this case, we must have $|B'(z_0)|<1$.
\item $B$ is called {\it hyperbolic} if the Denjoy-Wolff point $z_0$ of $B$ lies on $\partial \D$ and $B'(z_0)<1$,
\item $B$ is called {\it parabolic} if the Denjoy-Wolff point $z_0$ of $B$ lies on $\partial \D$ and $B'(z_0)=1$,
\end{enumerate}
Note that a fixed point on $\partial \D$ has real multiplier because $\partial \D$ is completely invariant under $B$.

Blashcke products are essentially determined by the location of their critical points. By a result of Zakeri \cite{Zakeri}, if $B_1, B_2$ share the same critical points counting multiplicity, then
$B_1=A\circ B_2$ for some M\"obius map $A$.

\subsection{Dynamics of Blaschke products}

For finite Blaschke products, it is not hard to see that the Julia set is always contained in $\partial \D$ and is either the whole of $\partial \D$ or a Cantor subset of $\partial \D$. These two cases can be characterized as follows; see \cite[p.58]{CG} and \cite{CDP} for a discussion of this characterization:

\begin{enumerate}[(i)]
\item if $B$ is elliptic, $J(B) = \partial \D$,
\item if $B$ is hyperbolic, $J(B)$ is a Cantor subset of $\D$,
\item if $B$ is parabolic and $z_0 \in \partial \D$ is the Denjoy-Wolff point of $B$, $J(B) = \partial \D$ if $B''(z_0) = 0$ and $J(B)$ is a Cantor subset of $\partial \D$ if $B''(z_0)\neq 0$.
\end{enumerate}

The class of Blaschke products of interest in this paper is the class of unicritical Blaschke products, namely those with one critical point in $\D$.
For $d\geq 2$, we define the set 
\[ \mathcal{B}_d  = \left \{ B_w(z):= \left ( \frac{z-w}{1-\overline{w}z} \right )^d : w\in \D, \arg(w) \in \left [0,\frac{2\pi}{d-1} \right ) \right \}\]
of normalized unicritical Blaschke products of degree $d$, which is parameterized by the sector 
\[ S_d = \left \{ w\in \D: \arg(w) \in \left [0,\frac{2\pi}{d-1} \right ) \right \} .\]
We denote by $\mathcal{E}_d$ those parameters in $S_d$ which give elliptic unicritical Blaschke products and by $\mathcal{C}_d$ those parameters in $S_d$ which give rise to unicritical Blaschke products with connected Julia set.
We further denote by $\widetilde{\mathcal{E}_d}\subset \D$  the set 
\[ \widetilde{\mathcal{E}_d} = \bigcup _{j=0}^{d-2} R_j(\mathcal{E}_d),\]
where $R_j$ is the rotation through angle $2\pi j/(d-1)$, and denote by $\widetilde{\mathcal{C}_d}$ the corresponding set for $\mathcal{C}_d$.  When $d=2$, $\mathcal{E}_d$ and $\widetilde{\mathcal{E}_d}$ agree.

The motivation for studying unicritical Blaschke products originally arose in classifying the dynamics of
quasiregular mappings of constant complex dilation near fixed points \cite{Fletcher2}. In \cite{Fletcher}, it was shown that every unicritical Blaschke product of degree $d$ is conjugate to a unique element of $\mathcal{B}_d$ and that the connectedness locus $\mathcal{C}_d$ consists of $\mathcal{E}_d$ and one point on the relative boundary in $S_d$ where $|w| = \frac{d-1}{d+1}$. Further, the set $\widetilde{\mathcal{E}_d} \subset \D$ is a star-like domain about $0$ which contains the disk $ \{w\in \D : |w| <\frac{d-1}{d+1} \}$ and 
the set $\widetilde{\mathcal{C}_d}$ consists of $\widetilde{\mathcal{E}_d}$ and $d-1$ points on its relative boundary in $\overline{\D}$.

In this paper, we will exactly specify the curve bounding $\widetilde{\mathcal{E}_d}$.

\section{Epicycloidal curves in parameter space}

Fix $d\geq 2$.
Recall the epicycloid curve given by \eqref{eq:epi}. We will take the larger circle to have radius $\frac{d-1}{d+1}$ and the smaller circle to have radius $\frac{1}{d+1}$. The corresponding epicycloid will then be contained in the closed unit disk and have parametric equations
$$ 
x(\theta) = \frac{  \cos (d\theta)-d\cos\theta }{d+1} ,\quad y(\theta) = \frac{ \sin(d \theta)-d\sin \theta}{d+1},
$$
which can be written more concisely as $z(\theta) = \frac{ e^{id\theta}-de^{i\theta}}{d+1}$. We will  denote it by $\gamma_d$. 

We first give a proof of Proposition \ref{prop:multibrot} on epicycloids in Multibrot sets.

\begin{proof}
Fix $d\geq 2$ and consider the family of mappings $g_c(z) = z^d+c$ for $c\in \C$. The central hyperbolic component of the Multibrot set consists of those $c$ for which $g_c$ has an attracting fixed point. By continuity, this means that the boundary consists of those $c$ for which $g_c$ has a neutral fixed point.

Let $g_c(z_0) = z_0$ and $|g_c'(z_0)| = 1$. Then 
\[ d|z_0|^{d-1} = 1,\]
and so $z_0 = r_de^{i\alpha}$, where $r_d = d^{1/(1-d)}$ and where $\alpha$ is determined up to a $(d-1)$'th root of unity. Since $g_c(z_0) =  z_0$, we obtain
\[ c = z_0 - z_0^d = r_de^{i\alpha} - r_d^de^{id\alpha}.\]
Now $r_d^d = d^{d/(1-d)} = d^{-1+1/(1-d)} = r_d/d$ and so
\[c = \frac{r_d}{d} \left (  de^{i\alpha} - e^{id\alpha} \right ).\]
This is the equation of an epicycloid, where we starting with the smaller circle touching on the right. The ambiguity in determining $\alpha$ just corresponds to rotating the epicycloid through a multiple of $2\pi / (d-1)$.
\end{proof}

We next turn to unicritical Blaschke products.

\begin{theorem}
\label{thm:main}
The relative boundary of $\widetilde{\mathcal{E}_d}$ in $\overline{\D}$ is the epicycloid given by $\gamma_d$.
\end{theorem}

Since the relative boundary in $\overline{\D}$ is given by parabolic Blaschke products, it follows that this theorem is equivalent to showing that if $B$ is a degree $d$ unicritical Blaschke product in $\mathcal{B}_d$ with parameter $w$, then $w \in \gamma_d$. Note that we cannot have $|w|=1$, and so these points of $\gamma_d$ are excluded.

\begin{proof}
Suppose that $d\geq 2$ and $B\in \mathcal{B}_d$ is parabolic with parameter $w=se^{i\psi}$.
This means that the Denjoy-Wolff point $z_0$ of $B$ is on $\partial \D$ and satisfies $B'(z_0)=1$ and $B(z_0) = z_0$.
We therefore have
\begin{equation}
\label{eq:1}
\frac{d(1-|w|^2)}{|1-\overline{w}z_0|^2}=1,
\end{equation}
and
\begin{equation}
\label{eq:2}
\left (\frac{z_0-w}{1-\overline{w}z_0} \right)^d=z_0.
\end{equation}
The aim is to eliminate $z_0$ from these equations and have an expression for $w$.
To that end, we note that since $z_0\overline{z_0}=1$, it follows from \eqref{eq:1} that 
\[ d(1-s^2)=|1-\overline{w}z_0|^2=(|z_0|\cdot|\overline{z_0}-\overline{w}|)^2=|\overline{z_0}-\overline{w}|^2=|z_0-w|^2, \] 
and so 
\[ |z_0-w|=\sqrt{d(1-s^2)}.\]
Taking polar coordinates of $z_0$ with respect to $w$, see Figure \ref{fig:2} we have
\begin{equation}
\label{eq:3}
z_0-w=re^{i\phi},
\end{equation}
where $r=\sqrt{d(1-s^2)}$. 

\begin{figure}[h]
\begin{center}
\includegraphics[width=3in]{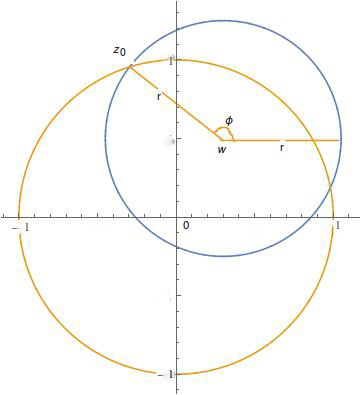}
\caption{Polar coordinate of $z_0$ with respect to $w$.}\label{fig:2}
\end{center}
\end{figure}

Since  $z_0$ is on the unit circle, we get that
\[ 1=z_0\overline{z_0}=(w+re^{i\phi})(\overline{w}+re^{-i\phi}) = s^2+wre^{-i\phi}+\overline{w}re^{i\phi}+r^2,\]
from which we obtain
\[ wre^{-i\phi}+\overline{w}re^{i\phi} = 1-s^2-r^2 = \frac{(1-d)r^2}{d},\]
where we have used $ r^2 = d(1-s^2)$.
If $ r=0 $ then $s=1$, which is impossible since $w\in \D$. Hence
\begin{equation}
\label{eq:4}
r=\frac{d}{1-d}(we^{-i\phi}+\overline{w}e^{i\phi}).
\end{equation}

On the other hand, using $z_0\overline{z_0}=1$ again, we can see from \eqref{eq:2} that
\[ z_0=\left (\frac{z_0-w}{1-\overline{w}z_0} \right)^d= \left (\frac{z_0-w}{z_0(\overline{z_0}-\overline{w})} \right)^d= \left (\frac{1}{z_0} \right)^d \left (\frac{z_0-w}{\overline{z_0}-\overline{w}} \right)^d, \]
and so we conclude
\[ \left (\frac{z_0-w}{\overline{z_0}-\overline{w}} \right)^d=z_0^{d+1}.\]
It follows from \eqref{eq:3} and by taking $(d+1)$'th roots that
\[\eta \exp \left ( \frac{2id\phi}{d+1} \right ) =w+re^{i\phi},\]
for some $(d+1)$'th root of unity $\eta$. To evaluate $\eta$, by continuity of $\phi$ in $w$ it is enough to evaluate this expression at one point. It is not hard to check that if $w=(1-d)/(1+d)$, then $B_w(1)=1$ and $B_w'(1)=1$ and so $z_0=1$ is the corresponding Denjoy-Wolff point. Hence $\phi=0$ for this $w$ and so we may take $\eta =1$.

Substituting \eqref{eq:4} into this equation, we obtain
\[\exp \left ( \frac{2id\phi}{d+1} \right ) = \frac{w}{1-d}+\frac{de^{2i\phi}\overline{w}}{1-d},\]
and so
\begin{equation}
\label{eq:5}
w=(1-d)e^{2id\phi/(d+1)}-de^{2i\phi}\overline{w}.
\end{equation}
Conjugating \eqref{eq:5} yields
\[ \overline{w}=(1-d)e^{-2id\phi/(d+1)}-de^{-i2\phi}w,\]
and substituting this back into \eqref{eq:5} gives
\[ w = (1-d)e^{2id\phi/(d+1)}-de^{2i\phi}\left ( (1-d)e^{-2id\phi/(d+1)}-de^{-i2\phi}w \right ).\]
Solving for $w$, we find that
\[ w = \frac{ e^{2id\phi/(d+1)}}{d+1} - \frac{de^{2i\phi/(d+1)}}{d+1}.\]
Letting $\theta=\frac{2\phi}{d+1}$, we obtain
\[ w = \frac{e^{id\theta} - de^{i\theta}}{d+1}.\]
This is nothing other than the parametric equation for the epicycloid $\gamma_d$, which completes the proof.
\end{proof}

To explain the geometric significance of $\theta$ and locate the Denjoy-Wolff point for any given
point on $\gamma_d$,  we prove the following.

\begin{corollary}
\label{cor:1}
Let $w = (e^{id\theta} - de^{i\theta})/(d+1)\in \gamma_d \cap \D$ and let $z_0$ be the Denjoy-Wolff point of $B_w$. Then $z_0 = e^{id\theta}$.
\end{corollary}

\begin{proof}
We obtain from \eqref{eq:3} and \eqref{eq:4} that, noting $\phi=(d+1)\theta/2$,
$$
z_0 =  \frac{w}{1-d}+\frac{de^{2i\phi}\overline{w}}{1-d}=\frac1{1-d}\left(\frac{e^{id\theta} - de^{i\theta}}{d+1}+ de^{2i\phi}\frac{e^{-id\theta} - de^{-i\theta}}{d+1}\right)=e^{id\theta}.
$$
\end{proof}

This means that $\theta =0$ corresponds to $w = (1-d)/(1+d)$ and $z_0 =1$. As $\theta$ moves from $0$ to $2\pi/(d-1)$, that is, from the cusp at $(1-d)/(1+d)$ to the next cusp $w_1$ taken anticlockwise, the Denjoy-Wolff point $z_0$ moves through an angle of $2\pi d/(d-1)$. This means $z_0$ completes one full revolution of the circle and in addition moves opposite the cusp $w_1$. 

In conclusion, as $w$ completes one revolution, the Denjoy-Wolff point $z_0$ completes $d$ revolutions, and we can consider $\theta$ geometrically as $\arg(z_0)/d$, see Figure \ref{fig:3}.

\begin{figure}[h]
\begin{center}
\includegraphics[width=4in]{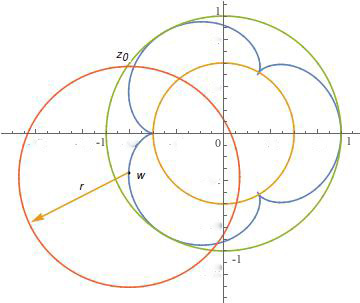}
\caption{How $w$ and $z_0$ vary as $\theta$ varies between $0$ and $2\pi / (d-1)$.}\label{fig:3}
\end{center}
\end{figure}

\section{The case $d=2$}

In this section, we discuss the whole family of degree two Blaschke products. Since a degree two Blaschke product has one critical point in $\D$, it follows that it is unicritical and hence conjugate by a M\"obius mapping to a Blaschke product in $\mathcal{B}_2$. Now, every degree two Blaschke product is of the form
\begin{equation}\label{eq:B2} B(z) = e^{it} \left ( \frac{z-w}{1-\overline{w}z} \right)\left ( \frac{z-u}{1-\overline{u}z} \right),\end{equation}
for some $t \in [0,2\pi )$ and $u,w\in \D$.
If $R$ is the rotation $R(z) = e^{it}z$, then
\[ RBR^{-1}(z) = e^{2it} \left ( \frac{ze^{-it}-w}{1-\overline{w}ze^{-it}} \right)\left ( \frac{ze^{-it}-u}{1-\overline{u}ze^{-it}} \right) = \left ( \frac{z-we^{it}}{1-\overline{we^{it}}z} \right)\left ( \frac{z-ue^{it}}{1-\overline{ue^{it}}z} \right).\]
Therefore, conjugating by a rotation means we can get rid of the $e^{it}$ term in \eqref{eq:B2}. From now on, we assume that 
\begin{equation}\label{eq:B3} B(z) =  \left ( \frac{z-w}{1-\overline{w}z} \right)\left ( \frac{z-u}{1-\overline{u}z} \right)\end{equation}
for some $u,w\in \D$.

\subsection{Location of the critical point}

It is well-known that the critical points of a Blaschke product lie inside the hyperbolic convex hull of the zeros of the Blaschke product, see \cite{Walsh}. For degree two Blaschke products, we can even say a little more. This is presumably well-known, but, for the convenience of the reader,  we include it in the proof of the following theorem.

\begin{theorem}
\label{thm:2}
Let $B$ be a Blaschke product of the form \eqref{eq:B3}. Then the unique critical point $c$ of $B$ in this unit disk lies at the midpoint of the hyperbolic geodesic between $u$ and $w$.
Further, there are sets $D_1, D_2$ such that $D_1\cup D_2=\D$, $D_1\cap D_2=\{c\}$ and
$B$ is a one-to-one map from $D_j$ ($j=1, 2$)  onto $\D$.
\end{theorem}

\begin{proof}
Let $B$ be as in \eqref{eq:B3} with zeros at $u,w$ and critical point at $c$ on the hyperbolic geodesic joining $u$ to $w$. Let 
\[ p(z) = \frac{z-c}{1-\overline{c}z}\]
be a M\"obius transformation sending $c$ to $0$. Define $\widetilde{B} = p\circ B \circ p^{-1}$. Then $\widetilde{B}$ is another degree two Blaschke product with critical point at $0$. By \cite[Corollary 2]{Zakeri}, there exists a M\"obius map $A$ such that $\widetilde{B}(z) = A(z^2)$. Hence $A^{-1} \circ p \circ B (z)= [p(z)]^2$.

Now since $u,w$ are both zeros of $B$, we have $A^{-1}(p(B(u))) = A^{-1}(p(B(w))) $, but this means $[p(u)]^2 = [p(w)]^2$. Since $p(u) \neq p(w)$, we have $p(u) = -p(w)$. In other words $p(c) =0$ is the hyperbolic midpoint between $p(u)$ and $p(w)$ and applying $p^{-1}$ gives the first result.

To prove the second part, we choose a M\"obius map $q$ with $q(c)=0$, $q(u)=-q(w)$ and $q(u) \in \R^+$.
Thus, $q\circ B$ only has one critical point at $0$ in $\D$. 
Again, by \cite[Corollary 2]{Zakeri}, there exists a M\"obius map $S$ such that $S\circ B \circ q^{-1}(z) = z^2$. 
Set
$$
E_1=\{z \in \D: \ \mbox{Im}z>0 \ \mbox{or} \ 0\le z <1\} \quad \mbox{and} \quad 
E_2=\{z \in \D: \ \mbox{Im}z<0\ \mbox{or} \  \ -1< z\le 0\}.
$$
Obviously, $E_1\cup E_2=\D$, $E_1\cap E_2=\{0\}$, and $z^2$ is a one-to-one  map from $E_j$ ($j=1,2$) onto $\D$, and consequently, 
so is $B$ from $D_j\stackrel{def}{=}q^{-1}(E_j)$ onto $\D$. It follows that the unit disk is divided by the hyperbolic
geodesic passing through points $c, u, w$. 
\end{proof}

The hyperbolic midpoint of $u$ and $w$ plays an important role in what follows. It can be explicitly computed in terms of $u$ and $w$.

\begin{proposition}
\label{prop:2}
If $w=-u$, then $c=0$. Otherwise, the hyperbolic midpoint of $u$ and $w$ is
\[ c = \frac{ |wu|^2 - 1 + \sqrt{ (1-|w|^2)(1-|u|^2)|1-w\overline{u}|^2} }{\overline{wu}(w+u) - \overline{(w+u)} }.\]
\end{proposition}

\begin{proof}
By computing $B'(c) = 0$, we find that $c$ must satisfy the quadratic equation
\begin{equation}\label{new.eq1}
(\overline{wu}(w+u)-\overline{(w+u)})c^2+2(1-|wu|^2)c+wu\overline{(w+u)}-(w+u)=0.
\end{equation}
If $w=-u$, we have $\overline{wu}(w+u)-\overline{(w+u)}=0$ and so $c=0$.
On the other hand, if $w\neq -u$, then there are two possible solutions of the quadratic:
\begin{equation*}
c_1=\frac{|wu|^2-1 + \sqrt{(1-|wu|^2)^2-|\overline{wu}(w+u)-\overline{(w+u)}|^2}}{\overline{wu}(w+u)-\overline{(w+u)}},
\end{equation*}
and
\begin{equation*}
c_2=\frac{|wu|^2-1 - \sqrt{(1-|wu|^2)^2-|\overline{wu}(w+u)-\overline{(w+u)}|^2}}{\overline{wu}(w+u)-\overline{(w+u)}}.
\end{equation*}
One may check that
\[
(1-|wu|^2)^2-|\overline{wu}(w+u)-\overline{(w+u)}|^2= (1-|w|^2)(1-|u|^2)|1-w\overline{u}|^2
\]
and so we find that for $u, w \in \D$,
$$(1-|wu|^2)^2-|\overline{wu}(w+u)-\overline{(w+u)}|^2>0.$$
Clearly, $|c_2|>|c_1|$ since $c_1$ and $c_2$ have the same denominator. Further,  (\ref{new.eq1}) gives
$$|c_1c_2|=\left| \frac{wu\overline{(w+u)}-(w+u)}{\overline{wu}(w+u)-\overline{(w+u)}}\right|=1,$$
and so we obtain $|c_1|<1$. Hence $B(z)$  has only one critical point in $\D$, and it is given by $c=c_1$.
\end{proof}

\subsection{Parameter for degree $2$ Blaschke products}

We next define a function which encodes the dynamics of $B$ in terms of the parameters.

\begin{definition}
\label{def:L}
Let $B$ be a Blaschke product of the form \eqref{eq:B3} with zeros at $u$ and $w$ and critical point at $c$. Then we define
\[ \lambda = \lambda(B) = \exp \left [ i \left ( \pi + 2\arg (1-\overline{c}w) + 2\arg ( 1-\overline{c} u) + 2\arg ( 1-c\overline{B(c)} \right ) \right ] \left ( \frac{B(c) - c}{1-\overline{c}B(c)} \right ).\]
\end{definition}

When $d=2$, no two distinct Blaschke products of the form $\left ( \frac{z-w}{1-\overline{w}z} \right )^2$ are conjugate by M\"obius transformations. In other words, $\mathcal{B}_2$ is parameterized by the unit disk $\D$. Recall that $\mathcal{E}_2$ is the set of parameters giving rise to elliptic Blaschke products in $\mathcal{B}_2$.

\begin{theorem}
\label{thm:3}
Let $B$ be of the form \eqref{eq:B3}. Then $B$ is elliptic, parabolic or hyperbolic respectively if and only if $\lambda \in \mathcal{E}_2$, the relative boundary of $\mathcal{E}_2$ in $\D$ or in $\D \setminus \overline{\mathcal{E}_2}$ respectively.
\end{theorem}

We remark that if we view $u$ as fixed and let $w\to u$, then $c\to u$, $1-\overline{c}w \to 1-|w|^2 \in \R$, $B(c) \to 0$ and so $\lambda \to e^{i\pi}(-c) = u$. Hence if the zeros of $B$ coincide, we just obtain again results from \cite{Fletcher}.

\begin{proof}[Proof of Theorem \ref{thm:3}]
Consider again the set-up in the proof of Theorem \ref{thm:2}. The Blaschke product $B$ is conjugate via $p$ to $\widetilde{B}$ and we may write $\widetilde{B}(z) = A(z^2)$. If we write
\[ A(z) = e^{i\sigma} \left ( \frac{z-\mu}{1-\overline{\mu}z} \right),\]
with $\sigma,\mu$ to be determined, then
\[ A^{-1} \circ \widetilde{B} \circ A(z) = [A(z)]^2 = e^{2i\sigma} \left (\frac{z-\mu}{1-\overline{\mu}z} \right )^2=: B_1(z).\]
Finally, if $R(z) = e^{2i\sigma}z$, we have
\[ R\circ B_1 \circ R^{-1}(z) = \left ( \frac{z-\mu e^{2i\sigma}}{1-\overline{\mu e^{2i\sigma}}z } \right )^2,\]
and this is a Blaschke product in $\mathcal{B}_2$. Since conjugating by M\"obius maps does not change the classification of Blaschke products, $B$ is classified by the location of $\mu e^{2i\sigma}$ in $\D$. This will be our $\lambda$, and so we just have to express this in terms of $u,w$ and $c$.

First, $A(0) = -\mu e^{i\sigma}$, but also since $A(z^2) = p(B(p^{-1}(z)))$, which is from the proof of Theorem \ref{thm:2},  we have 
\begin{equation}\label{eq:thm3a}
-\mu e^{i\sigma} = p(B(c)). 
\end{equation}
We also have
\begin{equation}
\label{eq:thm3b} 
A(p(w)^2) =e^{i\sigma} \left ( \frac{p(w)^2 - \mu }{1-\overline{\mu}p(w)^2} \right ) = p(B(w)) = -c.
\end{equation}
Therefore combining \eqref{eq:thm3a} and \eqref{eq:thm3b} yields
\[ \frac{1}{\mu} \left (  \frac{p(w)^2 - \mu }{1-\overline{\mu}p(w)^2} \right ) = \frac{c}{p(B(c))},\]
and rearranging gives
\[ p(w)^2 - \mu = \frac{c}{p(B(c))} \left ( \mu -|\mu |^2p(w)^2 \right).\]
Since 
$$|\mu| =|A(0)|=|p(B(p^{-1}(0)))|= |p(B(c))|,$$
and so
\[ p(w)^2 - \mu = \frac{c}{p(B(c))} \left ( \mu -|p(B(c)) |^2p(w)^2 \right).\]
Solving for $\mu$, we obtain
\[ \mu = p(w)^2p(B(c)) \left ( \frac{1+ c\overline{p(B(c))}}{c+p(B(c))} \right ).\]
Using \eqref{eq:thm3a} again gives $e^{i\sigma}=p(B(c))/(-\mu)$,
therefore, the parameter $\lambda$ is given by
\begin{align*} 
\lambda &= \mu e^{2i\sigma} = \frac{p(B(c))^2}{\mu} \\
&=p(w)^{-2} p(B(c)) \left ( \frac{c+p(B(c))}{1+c\overline{p(B(c))}} \right ).
\end{align*}
This can be simplified as follows. First of all, we have
\begin{align*}
\frac{c+p(B(c))}{1+c\overline{p(B(c))}}  &= \frac{ c+ \frac{B(c) -c}{1-\overline{c}B(c)} }{1+c \left ( \frac{ \overline{B(c)} - \overline{c} }{1-c\overline{B(c)}} \right ) }\\
&= \frac{ (1-c\overline{B(c)})B(c)(1-|c|^2) }{(1-\overline{c}B(c))(1-|c|^2)}\\
&= B(c)\exp[ 2i\arg(1-c\overline{B(c)})].
\end{align*}
Next, since $p(w) = -p(u)$ we have
\begin{align*}
B(c)p(w)^{-2} &= -\left ( \frac{ c-w}{1-\overline{w}c} \right)\left ( \frac{ c-u}{1-\overline{u}c} \right)\left ( \frac{ w-c}{1-\overline{c}w} \right)^{-1}\left ( \frac{ u-c}{1-\overline{c}u} \right)^{-1}\\
&= -\exp[2i\arg(1-\overline{c}w)]\exp[2i\arg(1-\overline{c}u)]\\
&= \exp[i( \pi + 2\arg (1-\overline{c}w)+2\arg(1-\overline{c}u))].
\end{align*}
Subsituting this into the expression for $\lambda $ yields
\[ \lambda =  \exp \left [ i \left ( \pi + 2\arg (1-\overline{c}w) + 2\arg ( 1-\overline{c} u) + 2\arg ( 1-c\overline{B(c)} \right ) \right ] \left ( \frac{B(c) - c}{1-\overline{c}B(c)} \right ) \]
as required.
\end{proof}

The significance of $\lambda$ is that $B$ of the form \eqref{eq:B3} is conjugate to $B_{\lambda}(z) = \left ( \frac{z-\lambda}{1-\overline{\lambda}z} \right )^2$. Hence if two Blaschke products $B_1,B_2$ have the same parameter $\lambda$, then they are conjugate via a M\"obius transformation.

\subsection{The real case}

If we further restrict to the case where $u,w$ are real, we may say more.

\begin{lemma}
\label{lem:real}
If $u,w$ are real, then $\lambda$ is real.
\end{lemma}

\begin{proof}
Since $u,w$ are real, the hyperbolic midpoint $c$ is real. Further, $B(c)$ must be real. All the terms of the form $\arg(1-\overline{c}u)$ in Definition \ref{def:L} are $0$, and so we conclude that
\[ \lambda = \frac{c-B(c)}{1-cB(c)} \in \R.\]
\end{proof}

\begin{theorem}
\label{thm:real}
For $\lambda \in (-1,1)$, the conjugacy class of $B_{\lambda}$ in
\[ \left \{ \left ( \frac{z-w}{1-wz} \right ) \left ( \frac{z-u}{1-uz} \right ) : u,w \in (-1,1) \right \} \]
is parameterized by the curve $w=f_{\lambda}(u)$ in $(-1,1)^2$, where
\[ f_{\lambda}(u) = \frac{(1+\lambda)u - 2\lambda}{2u-(1+\lambda)}.\]
For $\lambda \in (-1,1)$, the curves $f_{\lambda}$ are distinct, decreasing and each intersects the line $u=w$ in only one point, at $u=w=\lambda$.
\end{theorem}

The significance of this theorem is that in the parameter space of $(u,w) \in (-1,1)^2$, each conjugacy class is given by a curve that can be collapsed onto the line $u=w$, which gives the real axis in the cardioid picture of parameter space for degree $2$ Blaschke products. See Figure \ref{fig:4}.

\begin{figure}[h]
\begin{center}
\includegraphics[width=2in]{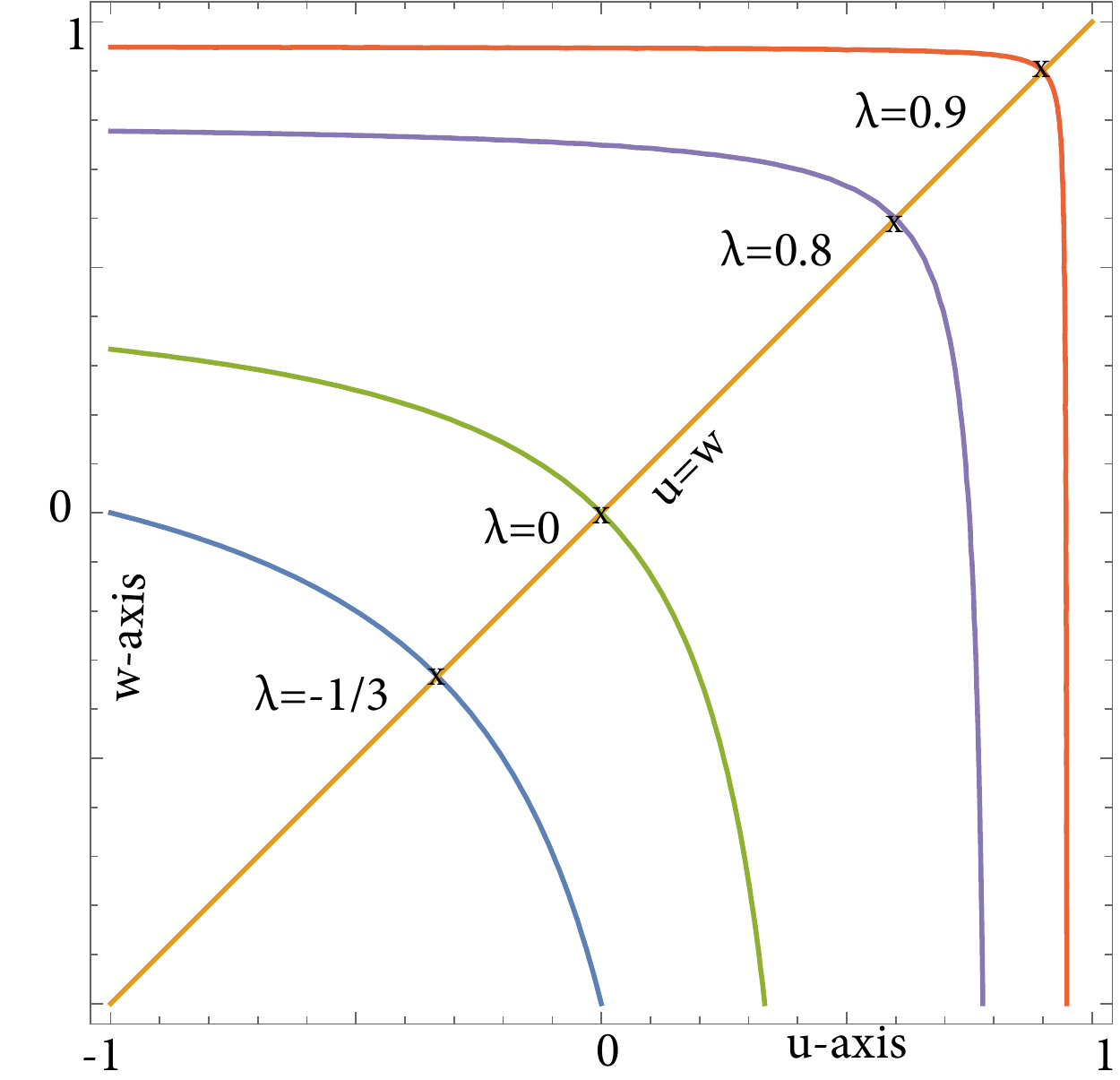}
\hspace{1in}
\includegraphics[width=1.7in]{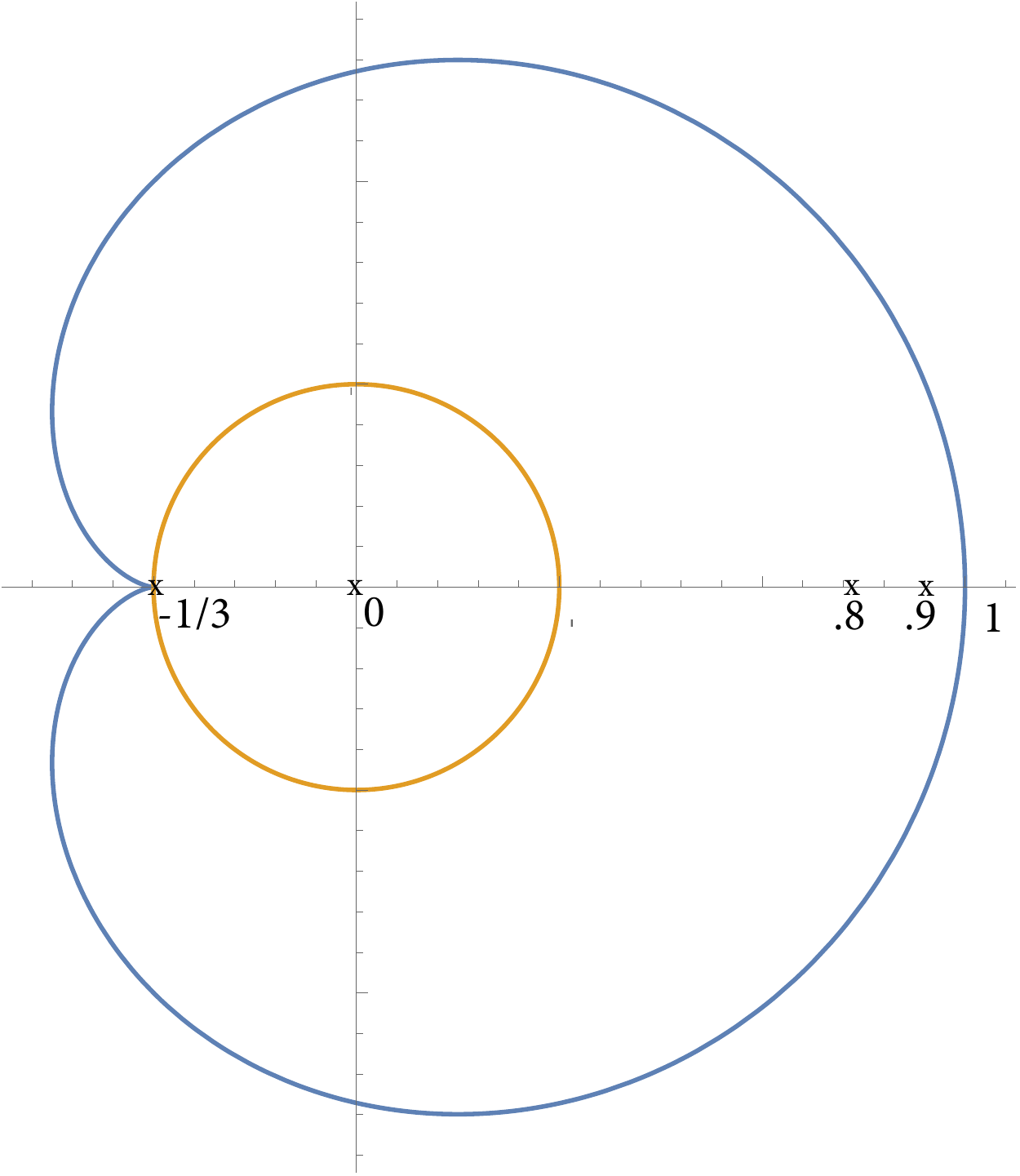}
\caption{How $(-1,1)^2$ relates to the parameter space for degree $2$ Blaschke products.}\label{fig:4}
\end{center}
\end{figure}

\begin{proof}
First, if $w=-u$, then $c=0$ and $\lambda =-B(0)=u^2$. Therefore, $w=f_{\lambda}(u)=-u$.
Now we assume that $w\neq -u$, i.e., $c\neq 0$.
By Lemma \ref{lem:real}, given $u,w$, we have $\lambda = \frac{c-B(c)}{1-cB(c)}$. 
Recalling from the proof of Theorem \ref{thm:2} the M\"obius map $p$ which sends $c$ to $0$, we have
\[ B(c) = \left ( \frac{c-w}{1-wc} \right ) \left ( \frac{c-u}{1-uc} \right ) = p(w)p(u) = -p(w)^2.\]
Therefore
\begin{equation}\label{new.eq2} 
\lambda = \frac{ c+p(w)^2}{1+cp(w)^2}.
\end{equation}
Recall from Proposition \ref{prop:2} that 
\[ c = \frac{ |wu|^2 - 1 + \sqrt{ (1-|w|^2)(1-|u|^2)|1-w\overline{u}|^2} }{\overline{wu}(w+u) - \overline{(w+u)} }.\]
Now when $u,w \in \R$, this simplifies to 
\begin{equation}\label{new.eq3}
 c = \frac{1+uw - \sqrt{(1-w^2)(1-u^2)}}{w+u}\stackrel{def}{=}\frac{1+uw - \sqrt{\Delta}}{w+u},
\end{equation}
where $\Delta=(1-w^2)(1-u^2)=1-u^2-w^2+u^2w^2$.
Next, since $p$ is a M\"obius transformation, it leaves the cross-ratio invariant. Applying this observation to 
\[ (w,u, 1/w, 1/u ) = ( p(w),-p(w),1/p(w), -1/p(w) ),\]
we see that
\[ \frac{ \Delta}{(1-uw)^2} = \left ( \frac{ 1-p(w)^2}{1+p(w)^2} \right )^2.\]
Solving for $p(w)^2$, we obtain
\begin{equation} \label{new.eq4}
p(w)^2 = \frac{(1-wu)-\sqrt{\Delta}}{(1-wu)+\sqrt{\Delta}}.
\end{equation}
It follows from (\ref{new.eq2}), (\ref{new.eq3}) and (\ref{new.eq4}) that
$$
 \lambda 
=\frac{(w+u)[(1-wu)-\sqrt{\Delta}]+[1-(wu-\sqrt{\Delta})^2]}{(w+u)[(1-wu)+\sqrt{\Delta}]+[(1-\sqrt{\Delta})^2-w^2u^2]}\stackrel{def}{=}\frac{I_1}{I_2}.
$$
Now we simplify $I_1$ and $I_2$. Indeed,
\begin{align*}
I_1+I_2& =2(w+u)(1-wu)+2-2w^2u^2+2(wu-1)\sqrt{\Delta} \\
&=2(1-wu)\left(w+u+1+wu-\sqrt{\Delta}\right),
\end{align*}
and
\begin{align*}
I_2&= w+u-w^2u-wu^2+(w+u)\sqrt{\Delta}+1-2\sqrt{\Delta}+\Delta-w^2u^2 \\
&=w(1-u^2)+u(1-w^2)+(1-w^2)+(1-u^2) -(2-w-u)\sqrt{\Delta}\\
&=(1-u^2)(w+1)+(1-w^2)(u+1)-(2-w-u)\sqrt{\Delta}\\
&=(1+u)(1+w)(1-w+1-u)-(2-w-u)\sqrt{\Delta}\\
&=(2-w-u)\left((1+u)(1+w)-\sqrt{\Delta}\right).
\end{align*}
Thus,
$$
\lambda=\frac{I_1}{I_2}=\frac{I_1+I_2}{I_2}-1=\frac{2(1-wu)}{2-w-u}-1=\frac{w+u-2wu}{2-w-u}.
$$

Rearranging in terms of $w$ we obtain
\[ w = \frac{(1+\lambda)u - 2\lambda}{2u-(1+\lambda)} =: f_{\lambda}(u).\]
It is elementary to check that the stated properties of $f_{\lambda}$ hold.
\end{proof}

\section{Acknowledgment}
The first author would like to thank Department of Mathematical Sciences, Northern Illinois University for  its hospitality during his visit.

\end{document}